\documentclass[12pt]{amsproc}
\usepackage{amssymb}

\usepackage{amsfonts,amsmath,amsthm}
\usepackage{amssymb}
\usepackage[utf8]{inputenc}
\usepackage[english]{babel}
\usepackage{graphicx}
\usepackage{array}
\usepackage{url}
\usepackage{hyperref}
\usepackage{booktabs}

\newtheorem{theorem}{Theorem}[section]
\newtheorem*{theorem*}{Theorem}

\newtheorem{proposition}[theorem]{Proposition}

\newtheorem{definition}[theorem]{Definition}
\theoremstyle{remark}

\newtheorem{example}[theorem]{Example}

\def\K{\mathbb K}
\def\Z{\mathbb Z}

\title[On the concept of strict convexity]{A precision on the concept of strict convexity in non-Archimedean analysis} 

\begin{document}

\author
{Javier~Cabello~Sánchez, José~Navarro~Garmendia*}
\address{*Corresponding author: Departamento de Matem\'{a}ticas, Universidad de Extremadura, 
Avenida de Elvas s/n, 06006; Badajoz. Spain}
\email{coco@unex.es, navarrogarmendia@unex.es}


\begin{abstract} We prove that the only non-Archimedean strictly convex spaces are the zero space and the one-dimensional linear space over $\, \Z/3\Z$, with any of its trivial norms.
\end{abstract}

\maketitle

\section{Introduction}

In order to find a non-Archimedean version of the Mazur-Ulam theorem, M. Moslehian and G. Sadeghi introduced in \cite{Moslehian1} the class of non-Archimedean strictly convex spaces.

Later on, A. Kubzdela observed that the only non-Archimedean strictly convex space over a field with a non-trivial valuation is the zero space (\cite[Theorem 2]{KubzdelaUno}). 

In this note, we prove:

\medskip
{\sc Proposition \ref{MenosUnoMasDos}} 
{\em The only non-Archimedean strictly convex spaces are the zero space and the one-dimensional linear space over $\, \Z/3\Z$, with any of its trivial norms.}
\medskip

\section{Non-Archimedean strictly convex spaces}

Firstly, recall that on a non-Archimedean normed space $X$, ``any triangle is isosceles"; 
that is to say, for any $x, y \in X$, 
$$ \| y \| < \| x \| \quad \Rightarrow \quad \| x + y \| = \| x \| \ . $$

\begin{definition}[\cite{Moslehian1}]\label{defnans} 
A non-Archimedean normed space $X$ over a field $\K$ is strictly convex if
\begin{enumerate}
\item[(SC1)] $|2|=1$. 
\item[(SC2)] for any pair of vectors $x,y \in X$, $\|x\|=\|y\|=\|x+y\|$ implies $x=y$.
\end{enumerate}
\end{definition}

Observe that (SC2) may also be rephrased by saying that there are no equilateral triangles; that is to say, that for any pair of distinct vectors $x \neq y \in X$, 
$$ \| y \| = \| x \| \quad \Rightarrow \quad \|x+y \| < \| x \|\ .$$

\begin{example} If $X$ is a one-dimensional normed linear space over 
a finite field, then its norm is trivial; i.e., there exists 
$a \in (0, \infty)$ such that $\|x\|= a $ for every nonzero vector $x \in 
X$. 

The one-dimensional linear space over $\Z / 3 \Z$, with any of its possible trivial norms, is a strictly convex space, in the sense of the Definition above.
\end{example}

\begin{proposition}\label{MenosUnoMasDos}
If $X$ is a non-zero strictly convex space, then it is linearly isometric to a one-dimensional normed space over $\, \Z/3\Z$. 
\end{proposition}

\begin{proof} First of all, observe that a strictly 
convex space can only occur in characteristic 3: for any vector $x\in X$,
$$\|2x\|=\|x\|=\|-x\| \, ;$$
as $2x+(-x)=x$, condition (SC2) implies that $2x=-x$; that is to say, $3x =0$ for any vector $x$, and we conclude that $3=0$ in $\K$. 

Now suppose there are two non-zero vectors $x,y\in X$ such that 
$y\neq \pm x$ and we will arrive to a contradiction.

Condition (SC1) implies that the characteristic of $\K$ is not 2, and, hence, $x+y \neq x-y$. Without loss of generality we may also assume that 
$$\|y\|\leq\|x\| \quad \mbox{ and } \quad \|x-y\|\leq\|x+y\| \ . $$

If $\|y \|  < \| x \| $, then $x+ y$ and $x-y$ are distinct vectors with the same norm, $\| x + y \| = \| x \| = \| x-y \|$, and whose sum mantains the norm
$$\|(x+y) + (x-y)\| = \|2x\|= \|x\| = \| x + y \|   \ , $$ in contradiction with (SC2).


If $\|y \| = \| x \| $, then (SC2) implies the absurd chains of inequalities 
$$\|x+y\| < \|x\| = \|2x\| = \|(x+y) + (x-y)\| \leq \max\{\|x+y\|, \|x-y\|\}\ , $$
$$\|x-y\| < \|x\| = \|2x\| = \|(x+y) + (x-y)\| \leq \max\{\|x+y\|, \|x-y\|\}\ . $$

\end{proof}

\section*{Acknowledgments}
{Supported in part by DGICYT projects MTM2016-76958-C2-1-P and PID2019-103961GB-C21 (Spain), ERDF funds and Junta de Extremadura programs GR-15152, IB-16056 and IB-18087.}

\bibliographystyle{plain}

\end{document}